\documentclass[12pt]{article}

\usepackage[english]{babel}
\usepackage{cmap}				
\usepackage{amsmath}
\usepackage{amsfonts,amssymb,mathrsfs,amscd,amsmath,comment}
\usepackage{url}
\newtheorem{theorem}{Theorem}%

\newtheorem{lemma}[theorem]{Lemma}

\newtheorem{corollary}[theorem]{Corollary}
\newtheorem{definition}{Definition}
\newtheorem{defn}[definition]{Definition}
\newtheorem{definition-proposition}[theorem]{Definition-Proposition}
\newtheorem{proposition}[theorem]{Proposition}

\newcommand{\R}{\mathbb{R}}
\newcommand{\rp}{\mathbb{RP}}
\newcommand{\Z}{\mathbb{Z}}

\newcommand{\Nn}{\mathbb{N}}

\newcommand{\PP}{\mathbb{P}}

\newcommand{\conj}{\operatorname{conj}}
\newcommand{\Sep}{\operatorname{Sep}}
\newcommand{\qed}{\square}
\newcommand{\sign}{\operatorname{sign}}

\def\C{\mathbb C}
\def\hC{\widehat\C}
\def\hR{\widehat\R}
\let\wt=\widetilde
\begin{document}

\date{}

\title{On elements of degree $g$ in the separating semigroup of a real curve of genus $g$.}
\author{G.\,B. Mikhalkin, S.\,Yu. Orevkov}

\maketitle

\begin{abstract}
Let $C$ be a separating real algebraic curve of genus $g$ whose real locus consists of
$r$ components and let $S=\Sep(C)$ be its separating semigroup. An element $\delta\in S$
is called {\em extendable} if $\delta+\Nn_0^r\subset S$, where $\Nn_0=\{n\in\Z\mid n\ge 0\}$.
We prove that an element of $S$ is extendable if it is realizable by a  separating 
mobile divisor
of degree $g$ and projective rank 1.
\end{abstract}




\section{Introduction}
By a real curve $(C,\conj)$ we mean a connected closed Riemann surface of genus $g$
endowed with an involution of complex conjugation $\conj$, which sends multiplication
of tangent vectors by $i$ to multiplication by $-i$. 
The fixed point set $\R C$ of such an involution is a collection of $r\le g+1$ smooth circles.

Real curves are divided into two types: those for which the set $C\setminus\R C$ is disconnected are of type I,
and the others are of type II. The Riemann surface $C$ is canonically oriented by its complex orientation. 
For curves of type I (also called {\em separating}), the real locus $\R C$ inherits two opposite
orientations from the two components of the complement $C\setminus\R C$. Thus, among the $2^r$ possible orientations
of the multicomponent curve $\R C$, there are two preferred orientations.
Each of these two opposite orientations
is called a {\em complex orientation} of $\R C$.

A rational function $f:C\to\hC:=\C\cup\{\infty\}$ is {\em real} if it is equivariant
with respect to $\conj$, i.e., if $f\circ\conj=\conj\circ f$. It is called {\em totally real}
if $f^{-1}(\hR)=\R C$, where $\hR=\R\cup\{\infty\}$. In this case 
the restriction $f|_{\R C}: \R C\to\hR$ is an oriented covering for
the standard orientation of $\hR$ and one of the complex orientations of $\R C$.
Such a covering has a {\em polydegree} $\delta=(d_1,\dots,d_r)\in\Nn^r$,
composed of the degrees $d_j$ of the restrictions of the covering to the components of $\R C$
(which we assume to be numbered).

Clearly, if $f$ is totally real, then $C$ is separating. Moreover, L.~Ahlfors \cite{refAhlfors}
showed that on a separating curve one can always find a totally real rational function.
Therefore, totally real functions, considered as morphisms $C\to\PP^1$,
are called in \cite{refKS} {\em separating morphisms}.
A divisor consisting only of real points is called {\em totally real},
and the divisor of zeros of a totally real rational function is called {\em separating}.
It is easy to see that a separating divisor cannot contain multiple points,
therefore, with some abuse of language, we will not always
distinguish between separating divisors and their supports.

M.~Kummer and K.~Shaw \cite{refKS} considered the set 
$$
    \Sep(C) = \{d(f)\mid\,f:C\to\hC\text{ is totally real}\} \subset\Z^r
$$
consisting of all polydegrees of separating divisors 
on a real curve $C$, and showed that it is closed under addition.
This subset is called the {\em separating semigroup} of the real curve $C$.
The separating semigroups of $M$-curves, hyperelliptic curves, curves of genus $\le 4$, and plane quintics are computed in
\cite{refKS,refMO,refOr19,refOr25}.

\begin{defn}
An element $\delta\in \Sep(C)$ is called {\em extendable} if 
$$
      \delta+\Nn_0^r\subset \Sep(C), \qquad\text{where}\quad  \Nn_0 = \{n\in\Z\mid n\ge 0\}.
$$
\end{defn}

Kummer and Shaw proved that if a separating divisor $D$ is nonspecial,
then its polydegree is an extendable element in $\Sep(C)$.%
\footnote{At first glance one could think that the sufficient extendibility condition in \cite[Corollary 3.4]{refKS}
could be applied to some special divisors. However, this is not so, as one can see in the proof of this corollary.}
The main result of our paper gives a sufficient condition for the extendibility of elements 
corresponding to the special case, i.e., the case $h^1((f)_0)>0$
(note that a separating divisor of degree $g$ is always special).
 
\begin{theorem}\label{mainthm}
Let $C$ be a separating real algebraic curve of genus $g$, and
$f:C\to\hC$ be a totally real rational function of degree $g$, 
such that its divisor of zeros $(f)_0$ has projective rank~$1$, i.e., $h^0((f)_0)=2$.
Then for all, except finitely many, points $y\in\hR$, the preimage $f^{-1}(y)$ defines a divisor $D=(f)_y$ with the following property.

For a sufficiently small arc $A\subset\R C$ centered at any point $p\in f^{-1}(y)$, and for any choice of points $p',p''\in A$ lying in different components of $A\setminus\{p\}$, 
the divisor $D-p+p'+p''$ is a separating nonspecial divisor of degree $g+1$. 
\end{theorem}

\begin{corollary}
Under the assumptions of Theorem \ref{mainthm}, the polydegree of the divisor $D$ is an extendable element of the semigroup $\Sep(C)$.
\end{corollary}


\section{Divisors of projective rank 1}
The following statement is an adaptation of an observation
made in \cite{refOrGAFA} for curves embedded in surfaces,
to the case of abstract real curves. 
Let $C$ be a separating real curve,
and $f:C\to\hC$ be a real rational function on it.
Let $\omega$ be a real holomorphic form on the curve $C$, and $(\omega)_0$ be its divisor of zeros.

\begin{proposition}\label{opp-or}
Suppose that $p',p''\in f^{-1}(0)$ are two distinct points such that the divisor $(\omega)_0-(f)_0$ is not effective,
but the divisor $(\omega)_0-(f)_0+p'+p''$ is effective,
i.e., at one of the points $p'$, $p''$
the order of zero of 
$f$ is one greater than the order of zero of 
$\omega$.
Then the divisors $(\omega)_0-(f)_0+p'$ and $(\omega)_0-(f)_0+p''$ are not effective,
i.e., at each of the points $p'$, $p''$
the order of zero of 
$f$ is one greater than the order of zero of 
$\omega$.
Moreover, the orientations on $\R C$ defined by the forms $\omega$ and $df$ coincide in a neighborhood of one of the points $p',p''$,
and are opposite in a neighborhood of the other point.
\end{proposition}

\noindent{\em Proof.}
The preimage $f^{-1}([-\varepsilon,\varepsilon])$, $\varepsilon>0$, defines a system of paths on $C$.
At the points of the preimage $f^{-1}(0)$ other than $p'$ and $p''$, the order of zero of $\omega$ is not less than the order of zero of $f$,
and hence the integral of the form $\omega$ over the corresponding components of the path system 
is $o(\varepsilon)$.
At one of the points $p'$, $p''$, or at both,
the order of zero of $\omega$ is one less than the order of zero of $f$; consequently, the integral of $\omega$
over the corresponding components of the path system is essential, i.e., of order $\varepsilon$. For the same reason,
at this point (at each of these points) the orders of zero
of the forms $\omega$ and $df$ are equal, and thus the orientations they define 
either coincide or are opposite.

By Abel's theorem, the integral of $\omega$ over the entire system of paths $f^{-1}([-\varepsilon,\varepsilon])$ is zero,
and hence there must be two essential integrals and their signs
in the total sum must be opposite. 
$\qed$\\

This statement is especially convenient to apply to mobile divisors of projective rank 1,
as shown by the following observation (Proposition~\ref{mob1}).
Starting from this observation, in this section it will not be important for us whether $C$ is separating or even a real curve.
Recall that a {\em mobile divisor} is an effective divisor $D$
whose linear system $|D|$ has no base points.

\newcommand{\Supp}{\operatorname{Supp}}

\begin{proposition}\label{mob1}
Let $D$ be an effective
divisor of projective rank 1 (i.e., $h^0(D)=2$) with no multiple points.
Then for any two distinct points $p',p''$ of the support of $D$ 
there exists a form $\omega$ vanishing on all of $D\setminus \{p',p''\}$, but not on all of $D$,
and such that $\omega(p')\ne 0$ and $\omega(p'')\ne 0$, if the divisor $D$ is mobile.

If $D$ is a real divisor, then the form $\omega$ can be chosen real.
\end{proposition}

\noindent{\em Proof.}
Set $D''=D-p'-p''$, $d=\deg D$, $d''=\deg D''$.
Note that $\dim\Omega^1(D)=h^0(K-D)$, where
$\Omega^1(D)$ is
the space of holomorphic forms vanishing on $D$. 
We need to prove that $\Omega^1(D'')\not\subset\Omega^1(D)$.
This fact follows from the Riemann--Roch theorem. Indeed,
since $h^0(D)=2$ and $h^0(D'')>0$, we have
\begin{equation}\label{eq.mob1}
\begin{split}
    h^0(K-D'') &= h^0(D'') + g - d'' - 1 = h^0(D'') + g - d + 1   \\ & > g - d + 1 = h^0(K-D) - h^0(D) + 2 = h^0(K-D).
\end{split}
\end{equation}
Hence $\dim\Omega^1(D'')>\dim\Omega^1(D)$, i.e.,
there exists a form $\omega\in\Omega^1(D'')\setminus\Omega^1(D)$.

Now let us prove that $\omega(p')=0$ implies $\omega(p'')=0$ if the divisor $D$ is mobile.
This can be derived either 
from Abel's theorem (as in the proof of Proposition \ref{opp-or}), or again
from the Riemann--Roch theorem. Indeed, set
$D'=D-p'$. From the mobility of $D$ it follows that $h^0(D')=h^0(D)-1$.
Therefore 
$h^0(K-D')=h^0(K-D)$ (cf.~\eqref{eq.mob1}), and since
$\Omega^1(D)\subset\Omega^1(D')$, we conclude that
$\Omega^1(D)=\Omega^1(D')$, i.e., $\omega(p'')=0$ implies $\omega(p')=0$ for $\omega\in\Omega^1(D'')$.

In the real case, the space $\Omega^1(D-p'-p'')$
and its subspace $\Omega^1(D)$ 
are real, and hence the form
$\omega$ can also be chosen real.
$\qed$\\

The converse statement also holds.

\begin{proposition}\label{1mob}
Suppose that $D$ is an effective mobile 
divisor without multiple points, such that for any two distinct points $p',p''\in\Supp(D)$ 
there exists a form $\omega$ vanishing on $D\setminus \{p',p''\}$, but not at $p'$.
Then the projective rank of the divisor $D$ is equal to one.
\end{proposition}

\noindent{\em Proof.}
The projective rank of a mobile divisor is positive. If it is greater than one, then for any point $p'\in\Supp(D)$, there exists a deformation of the divisor $D$ leaving the point $p'$ fixed. If $p''\in\Supp(D)$ is a point that moves under this deformation, then the existence of a form $\omega$ from the hypothesis of the theorem contradicts Abel's theorem (cf. the proof of Proposition \ref{mob1}).
$\qed$


\section{Separating divisors and holomorphic forms}
%
Now let the curve $C$ be real and separating.
If $D$ is a totally real divisor on $C$ (i.e., $\Supp(D)\subset\R C$), and
$D$ is mobile, without multiple points, then, as noted in \cite{refOr25}, Proposition \ref{opp-or} sometimes gives both necessary
and sufficient conditions for the divisor $D$ to be separating.
It turns out that this observation is most conveniently applied when $D$ has projective rank 1.

Note that the linear system of such a divisor $D$ defines a real rational function $f$ such that
$D=(f)_0$, and 0 is a regular value of this function. 
This function is defined uniquely up to an automorphism of $\hR$.
The increase of the function $f$ defines an orientation in a neighborhood of the preimage $f^{-1}(0)$.

\if01{
\begin{proposition}\label{propKS} {\rm(cf.~\cite[Proposition 2.11]{refKS})}
Let $f:C\to\hC$ be a real rational function with a totally real divisor of zeros $(f)_0$ without multiple points.
The orientation of a neighborhood of the preimage $f^{-1}(0)$, given by the form $df$,
extends to an orientation on all of $\R C$ if and only if $f$ is totally real.
\end{proposition}

\noindent{\em Proof.}
By definition, all $d=\deg f$ points of the totally real divisor are real. 
If the orientation extension condition is satisfied, then the degree of the map $f_{\R}:\R C\to\hR$ in the obtained
orientation of the curve $\R C$ is equal to $\pm d$. Thus, this is a totally real map.
Conversely, if the function $f$ is totally real,
then $df$ defines an orientation on $\R C$.
$\qed$

\begin{corollary}\label{corKS}
If the orientation of a neighborhood of the preimage $f^{-1}(0)$, defined by 
the form $df$, extends to an orientation on all of $\R C$, then it coincides 
with one of the complex orientations of the curve $\R C$.
\end{corollary}

\noindent{\em Proof.}
 By Proposition \ref{propKS}, the function $f$ is totally real.
 Therefore, the orientation on $\R C$ defined by the form $df$ is a complex orientation,
 since it coincides with the boundary
 orientation induced from the preimage of the upper half-plane.
$\qed$\\

Putting the proven statements together, we obtain the following two facts.
}\fi

\begin{proposition}\label{or1}
Let $D$ be a separating divisor, 
and $p',p''\in\Supp(D)$ two points. 
Then any holomorphic form vanishing on all of $\Supp(D)\setminus\{p',p''\}$, 
but not at $p'$, defines local orientations at the points $p'$ and $p''$ that are not coherent 
with the complex orientation of $\R C$.
\end{proposition}

\noindent{\em Proof.}
Let $D=(f)_0$, where $f$ is a totally real function.
By Proposition \ref{opp-or}, the orientations given by the forms $\omega$ and $df$ are not coherent at the points $p'$ and $p''$.
Moreover, the orientation given by the form $df$ for a totally real function $f$ is a complex orientation.
$\qed$

\begin{proposition}\label{or2}
Let $D$ be a mobile totally real divisor 
without multiple points.
Suppose that for any distinct 
points $p',p''\in\Supp(D)$ there exists a holomorphic form $\omega_{p',p''}$,
vanishing on $\Supp(D)\setminus\{p',p''\}$, but not at $p'$, which 
defines local orientations at the points $p'$ and $p''$ not coherent with the complex 
orientation of $\R C$. Then the divisor $D$ is separating.
\end{proposition}

\noindent{\em Proof.}
Since $D$ is mobile, $D=(f)_0$, where $f$ is a real rational function.
For any $p',p''\in D$, the orientation given by the form $\omega_{p',p''}$ is not coherent in neighborhoods of these points
with the complex orientation (by hypothesis), nor with the orientation given by the form $df$ (by Proposition \ref{opp-or}),
Consequently, the latter two orientations are coherent, and hence $f$, considered as a mapping between two oriented
(by the complex orientations) manifolds, has degree $\pm\deg(D)$.
Hence it is totally real. 
$\qed$\\

It is convenient to introduce the following notation.
Fix a complex orientation on $C$.
For a point  $p\in\R C$ and a real holomorphic form $\omega$ defined in a complex neighborhood of $p$, set
$\sign_\omega(p) = \sign\omega(v)$, where $v$ is a tangent vector to $\R C$ at the point $p$, positively
oriented with respect to the chosen complex orientation.


If $p$ is a simple zero of a real rational function $f:C\to\hC$, then
$\sign_{df}(p)$ is the local degree of $f|_{\R C}:\R C\to\hR$ at the point $p$.

\begin{proposition}\label{or3}
Let $f:C\to\hC$ be a real rational function with a totally real divisor of zeros $D=(f)_0$ 
without multiple points. The function $f$ is totally real if and only if
for any points $p',p''\in\Supp(D)$ one has
\begin{equation}\label{eq.or3}
                                  \sign_{df}(p')=\sign_{df}(p'').
\end{equation}
\end{proposition}

To verify condition \eqref{eq.or3}, we shall 
use the following 
fact, which follows directly from Proposition \ref{opp-or}. 

\begin{proposition}\label{or4}
Let $f$ and $D$ be as in Proposition \ref{or3}.
Let $p',p''\in\Supp(D)$ and $\omega$ be a real holomorphic form on $C$,
vanishing everywhere on $f^{-1}(0)\setminus\{p',p''\}$, but not at $p'$. Then condition \eqref{eq.or3}
is equivalent to condition
\begin{equation}\label{eq.or4}
                                  \sign_\omega(p')\ne\sign_\omega(p'').
\end{equation}
\end{proposition}


\section{Dubling points on separating special divisors of projective rank 1}

\begin{lemma}\label{finite}
Let $f:C\to\hC$ be a rational function of degree $d\le g$ such that the divisor $D=(f)_0$ has projective rank 1.
Then for all but finitely many values $y\in\hC$, there exists a holomorphic form $\omega_y$ 
having first-order zeros at all points of $f^{-1}(y)$.
\end{lemma}

\noindent{\em Proof.}
By the Riemann--Roch theorem, $h^0(K-D)=g-d+1\ge 1$. 
Consequently, there exists a divisor $E$ such that $K_y=E+(f)_y$ belongs to the canonical class for all $y$.
Therefore, if $y$ is a regular value of the function $f$ and $y\not\in f(\Supp(E))$, then the holomorphic form $\omega_y$,
corresponding to the divisor $K_y$, has the required property.
Indeed, in this case the divisors $(f)_y$ and $E$ have no common points,
hence, at each point of $f^{-1}(y)$, the multiplicity of zero of the form $\omega_y$ 
is equal to the multiplicity of zero of the function $f$,
and since $y$ is a regular value, these multiplicities are equal to one.
$\qed$\\


\noindent{\em Proof of Theorem \ref{mainthm}.}
Let $f:C\to\PP^1$ be a totally real rational function of degree $g$ and projective rank 1, and 
$D=(f)_y$, $y\in\R$.
Since then $h^1(D)=h^0(K-D)=1$, there exists a unique (up to multiplication by a constant)
holomorphic form $\omega_D$ vanishing on $\Supp(D)$. 
Note that, since the divisor $D$ is mobile, the point $p$ is not a base point of it, hence $h^0(D-p)<h^0(D)$,
and thus $h^1(D-p)=1$, i.e., any form vanishing on $\Supp(D)\setminus\{p\}$ is also 
proportional to the form $\omega_D$.
If the arc $A$ is sufficiently small, then this form does not vanish at $p'$, and thus the divisor $D^+=D-p+p'+p''$
of degree $g+1$ is nonspecial and has projective rank 1, as does the divisor $D$.

Let us prove that the divisor $D^+$ is mobile.
Indeed, if $q$ is a base point of it, then $h^0(D^+-q)=h^0(D^+)=2$, hence the divisor $D^+-q$
is special, i.e., there exists a holomorphic form $\omega_q$ vanishing everywhere on $\Supp(D^+)\setminus\{q\}$.
If $q=p''$, then $\omega_q$ vanishes on $\Supp(D)\setminus\{p\}$, but then, as shown above, $\omega_q$
cannot vanish at the point $p'$. The case $q=p'$ is analogous. Suppose now that $q\not\in\{p',p''\}$.
If such a form $\omega_q$ exists for any points $p'$, $p''$ arbitrarily close to $p$, then, passing to the
limit, we obtain a holomorphic form which vanishes everywhere on $\Supp(D)\setminus\{q\}$ and has a multiple
zero at $p$. As shown above, such a form is proportional to the form $\omega_D$, and by Lemma~\ref{finite}
we can assume that $\omega_D$ has a first-order zero at $p$.
The obtained contradiction shows that the divisor $D^+$ is mobile.
Let $f^+$ be a real rational function such that $D^+=(f^+)_0$.

Since the vector space of forms vanishing on $\Supp(D)$ is defined over $\R$, we may assume 
that the form $\omega_D$ is real.
Let $q\in\Supp(D)\setminus\{p\}$.
Consider the real holomorphic form $\omega$ for the pair of points $q$ and $p$
in the divisor $D$, provided by Proposition \ref{mob1}.
Since $D$ is separating,
by Propositions \ref{or3} and \ref{or4} we have
\begin{equation}\label{eq1}
     \sign_\omega(p) \ne \sign_\omega(q).
\end{equation}
The point $p''$ is a zero of the form $\wt\omega=\omega_D+\varepsilon\omega$ for some $\varepsilon\in\R$. 
If the arc $A$ is sufficiently small, then $\varepsilon$ is also small, and then
\begin{equation}\label{eq2}
     \sign_{\wt\omega}(p') = \sign_{\wt\omega}(p) = \sign_\omega(p)
\end{equation}
(since $p$ is a simple zero of the form $\omega_D$) and
\begin{equation}\label{eq3}
     \sign_{\wt\omega}(q) = \sign_{\omega}(q).
\end{equation}
From \eqref{eq1}--\eqref{eq3} it follows that 
$$
   \sign_{\wt\omega}(p')=\sign_\omega(p)\ne\sign_\omega(q)=\sign_{\wt\omega}(q).
$$
The form $\wt\omega$ vanishes on $\Supp(D^+)\setminus\{p',q\}$, but not at $q$.
Consequently, $\sign_{df^+}(p')=\sign_{df^+}(q)$ by Proposition \ref{or4}.
Similarly, $\sign_{df^+}(p'')=\sign_{df^+}(q)$.
Since $q$ is an arbitrary point of $\Supp(D^+)\setminus\{p',p''\}$, we conclude that the function $\sign_{df^+}$
is constant on $\Supp(D^+)$, and Theorem~\ref{mainthm} follows from Proposition \ref{or3}.
$\qed$


\section{Some examples}

\subsection{Hyperbolic quartic in $\PP^2$}

A real curve $C$ of degree $m$ in $\PP^2$ is called {\em hyperbolic},
if its linear projection onto $\PP^1$ from some point $p\in\rp^2\setminus\R C$
is a separating morphism of degree $m$. If $C$ is non-singular and $m=2k$, this is equivalent to the fact that
$\R C=X_1\cup \dots\cup X_k$, where the $X_j$ are embedded circles (ovals) in $\rp^2$ such that
$X_j$ for $j>1$ bounds a disk in $\rp^2$ containing $X_{j-1}$.

For a plane hyperbolic quartic $C$, it is shown in \cite[Example 3.8]{refKS} that the element
$(1,2)\in\Sep(C)$ is extendable.%
\footnote{ This result is stated in \cite{refKS} for a concrete curve but the proof is valid for
any plane hyperbolic quartic.}
The elements $(1,2)$ and $(2,2)$ are realizable by the linear projection from a point chosen on the inner oval
or inside it (i.e., in the disk it bounds). The elements $(3,2)$ and $(1,3)$ are realized in \cite{refKS}
by pencils of conics, and their extendibility follows from the extendibility of elements realizable by
non-special divisors proved in \cite[Lemma 3.2]{refKS}.
The realizability of $(1,2)$ and $(2,2)$ combined with the extendibility of $(3,2)$ of $(1,3)$
implies the extendibility of $(1,2)$.

Theorem \ref{mainthm} of the present paper directly implies the extendibility of $(1,2)$.


\subsection{Other examples}

All elements of degree $g$ of the separating semigroups computed in
\cite{refKS,refMO,refOr19,refOr25} are extendable and, almost for all of them, this fact can be proved
(and in \cite{refMO} is proved) by applying Theorem~\ref{mainthm} after a non-difficult checking that
the rank of the corresponding divisor of degree $g$ is equal to $1$. The only exclusion is the element
$g$ of the semigroup
$$
     \Sep(C)=(2\mathbb{N}) \cup \big(g+\mathbb{N}_0\big)
$$
for a non-maximal separating hyperelliptic curve $C$ of even genus $g\ge 4$.
Any separating divisor $D$ of degree $g$ on such a curve has the form $\pi^*(E)$,
where $\pi:C\to\PP^1$ is the hyperelliptic projection and $E$ is a totally real divisor of degree
$g/2$ on $\PP^1$ without multiple points. The projective rank of such a divisor is $g/2$, hence Theorem~\ref{mainthm}
does not apply to it for $g>2$. Nevertheless, the corresponding element of the separating semigroup is extendable,
and can be constructed by dubling a point (cf.~\S\ref{sect.doubling}) but with a specially chosen 
simultaneous perturbation of the other points.


\medskip\noindent
{\bf Remark.} We say that an extendable element $a \in\Sep(C)$ is {\em extremal} if it cannot be presented as
$b+c$, where $b \in\Sep(C)$ extendable, $b\ne a$, and $c \in\mathbb N_0^r$.
It is interesting that all the extremal extendable elements of the semigroups classified in
\cite{refKS,refMO,refOr19,refOr25} are of degree $g$ or $g+1$,
and the degree $g + 1$ appears only for $M$-curves and for hyperelliptic curves of odd degree.


\bigskip
Universit\'e de Gen\`eve, Math\'ematiques, rue du Conseil-G\'en\'eral 7--9,
1205 Gen\`eve, Suisse (G.M.)

Email address: grigory.mikhalkin@unige.ch

\bigskip
Steklov Mathematical Institute, ul. Gubkina 8, Moscow, Russia (S.O.)

IMT, Universit\'e de Toulouse, 118 route de Narbonne, 31062, Toulouse, France (S.O.)

Email address: orevkov@math.ups-tlse.fr

\end{document}